\documentclass[preprint,12pt]{elsarticle}
\usepackage{graphicx} 
\usepackage{amssymb}
\usepackage{amsthm}
\usepackage{graphics}
\usepackage{epsfig}
\usepackage{amssymb}
\usepackage{hyperref}
\usepackage{graphicx}
\usepackage{subfigure}
\usepackage{float}
\usepackage{mathrsfs}
\usepackage{color}
\usepackage[dvipsnames]{xcolor}
\usepackage[svgnames]{xcolor}
\usepackage[x11names]{xcolor}
\usepackage{tikz-network}
\usepackage{tikz}
\usepackage{tikz,pgfplots}
\usetikzlibrary{positioning, arrows.meta}
\usetikzlibrary{patterns}

\usepackage[ruled,vlined]{algorithm2e}
\usepackage{caption}
\DeclareCaptionLabelFormat{cont}{#1~#2\alph{ContinuedFloat}}
\usepackage{fullpage}
\usepackage{amsmath,bbm}
\usepackage{amsfonts}
\usepackage{enumerate}
\usepackage{enumitem} 
\usepackage{tikz}
\usepackage{array}
\usepackage{bigdelim}
\usepackage{nicematrix}
\usepackage{mathtools}
\usepackage{blkarray, bigstrut}
\usepackage{gauss}
\usepackage{xcolor}
\usepackage{tabularray}
\UseTblrLibrary{booktabs}

\definecolor{sayanta}{HTML}{CC5500}

\newcommand{\pf}{\noindent {\bf Proof: }}
\newtheorem*{theoremaux}{Theorem \theoremauxnum}
\gdef\theoremauxnum{1}

\newtheorem{lemma}{\bf Lemma}[section]
\newtheorem{theorem}{\bf Theorem}[section]

\newtheorem{proposition}[lemma]{\bf Proposition}

\newtheorem{remark}{\bf Remark}[section]
\theoremstyle{definition}
\newtheorem{example}{\bf Example}[section]

\journal{~}

\begin{document}

\begin{frontmatter}
\title{A generalization of the sum of element orders function}

\author[1]{Sayanta Chatterjee\corref{cor1}}
\ead{sayanta.math@gmail.com }

\author[1]{Angsuman Das}
\ead{angsuman.maths@presiuniv.ac.in}

\author[1]{Sauvik Poddar}
\ead{sauvikpoddar1997@gmail.com}

\address[1]{Department of Mathematics, Presidency University, 86/1 College Street, Kolkata 700073, India\\}
\cortext[cor1]{Corresponding author}

%

\begin{abstract}
In this paper, for a finite group $G$ and a positive integer $k$, we introduce the function $$\psi''_k(G)=\dfrac{\psi_k(G)}{|G|^{k+1}}, \quad \text{where} \quad \psi_k(G)=\sum_{g \in G}o(g)^k,$$ as the $k$-th generalization of the function $\psi''$, introduced by M. T{\u{a}}rn{\u{a}}uceanu. Different thresholds for cyclic, abelian, nilpotent and (super)solvable groups in terms of $k$ are obtained.
It is shown that these thresholds often outperform similar thresholds derived from other parameters like $\psi'$ and $\psi''$. Moreover, considering $\psi''_k$ as a $(0,1]$-valued function, we study its injectivity, surjectivity and density.
\end{abstract}

\end{frontmatter}

\section{Introduction}
Since its introduction in \cite{amiri2009sums} by Amiri et al., the sum of element orders $\psi(G)\coloneqq\sum_{g\in G}o(g)$, has played a prominent role in finite group theory. A foundational result, established in \cite{amiri2009sums}, asserts that if $G$ is a group of order $n$, then $\psi(G) \le \psi(\mathbb{Z}_n)$, with equality holding if and only if $G$ is cyclic. This naturally prompted the definition of the normalized function $\psi'(G)\coloneqq\psi(G)/\psi(\mathbb{Z}_n)$ in \cite{herzog_longobardi_exact}. A further related invariant, $\psi''(G)\coloneqq\frac{\psi(G)}{|G|^2}$, was introduced in \cite{tuarnuauceanu_israel} by T{\u{a}}rn{\u{a}}uceanu. In fact, this parameter can be interpreted as the probability of a particular event occurring in a finite group (see Remark \ref{probab-remark}). One of the principal applications of $\psi'(G)$ and $\psi''(G)$ is the derivation of sharp thresholds that guarantee the nilpotency or (super)solvability of $G$.

Independently, Amiri and Amiri \cite{jafarian2014sum} generalized $\psi(G)$ to $\psi_k(G)\coloneqq\sum_{g\in G}o(g)^k$, which sums the $k$-th powers of the element orders, and proved the analogous inequality $\psi_k(G) \le \psi_k(\mathbb{Z}_n)$ for $|G|=n$. The corresponding normalized version, $\psi'_k(G)$, was later introduced by Dey and Mondal \cite{dey2024exact} as the ratio $\psi_k(G)/\psi_k(\mathbb{Z}_n)$.

In this paper, for a finite group $G$ and a positive integer $k$, we define $\psi''_k(G)$ as
$$\psi''_k(G)\coloneqq\frac{\psi_k(G)}{|G|^{k+1}}.$$
This is a new invariant that simultaneously generalizes both $\psi''(G)$ and $\psi'_k(G)$ (see Remarks \ref{probab-remark} and \ref{normalized-remark}). Apart from the generalization, the new variant can also be thought of as the probability that $k$ randomly chosen elements $y_1,y_2,\ldots,y_k$ in $G$ belong to the cyclic group generated by another randomly chosen element $x\in G$. It follows from the argument given below: Let $n = |G|$. Choose $x \in G$ uniformly and independently choose 
$y_1,\dots,y_k \in G$ uniformly. For a fixed $x$, the subgroup 
$\langle x \rangle$ has exactly $o(x)$ elements, hence
$$\Pr\left(( y_1, \dots, y_k \in \langle x \rangle \mid x \right)
= \left( \frac{o(x)}{n} \right)^k.$$
Averaging over $x$ gives
$$\Pr\left( y_1, \dots, y_k \in \langle x \rangle \right)
= \frac{1}{n} \sum_{x \in G} \frac{o(x)^k}{n^k}
= \frac{1}{n^{k+1}} \sum_{x \in G} o(x)^k
= \frac{\psi_k(G)}{n^{k+1}}
= \psi''_k(G).$$
\begin{remark}\label{probab-remark}
    As putting $k=1$ in $\psi''_k(G)$ yields $\psi''(G)$ (introduced in \cite{tuarnuauceanu_israel}), $\psi''(G)$ can also be interpreted as the probability that a randomly chosen element $y$ in $G$ belongs to the cyclic group generated by another randomly chosen element $x\in G$. 
\end{remark}
\begin{remark}\label{normalized-remark}
    For a group $G$ of order $n$, it is easy to note that $$\psi'_k(G)=\dfrac{\psi''_k(G)}{\psi''_k(\mathbb{Z}_n)}.$$ Hence $\psi'_k(G)$ can be thought of as a normalized version of $\psi''_k(G)$.
\end{remark}


The threshold values for nilpotent, cyclic and (super)solvable groups for $\psi'$ are spanned over several articles. We put those results together in the following theorem.

\begin{theorem}\label{thresholds-for-psi-prime}
Let $G$ be a finite group. Then the following hold:
\begin{enumerate}
\item[$(a).$] \cite{herzog_longobardi_exact} If $\psi^{\prime}(G)>\psi^{\prime}(\mathbb{Z}_2\times\mathbb{Z}_2)=\frac{7}{11}$, then $G$ is cyclic.


\item[$(b).$] \cite{tuarnuauceanu2021criterion} If $\psi^{\prime}(G)>\psi^{\prime}(S_3)=\frac{13}{21}$, then $G$ is nilpotent.

\item[$(c).$] \cite{azad2022two} If $\psi^{\prime}(G)>\psi^{\prime}(A_4)=\frac{31}{77}$, then $G$ is supersolvable.

\item[$(d).$] \cite{azad2018criterion} If $\psi^{\prime}(G)>\psi^{\prime}(A_5)=\frac{211}{1617}$, then $G$ is solvable.
\end{enumerate}
\end{theorem}

In \cite{tuarnuauceanu_israel}, the author provided the thresholds for nilpotent, cyclic, abelian and (super)solvable groups for $\psi''$. In particular, he proved the following result.

\begin{theorem}\label{thresholds-for-psi-double-prime}\cite[Theorem 1.1]{tuarnuauceanu_israel}
Let $G$ be a finite group. Then the following hold:
\begin{enumerate}
\item[$(a).$] If $\psi^{\prime\prime}(G)>\psi^{\prime\prime}(\mathbb{Z}_2\times\mathbb{Z}_2)=\frac{7}{16}$, then $G$ is cyclic.

\item[$(b).$] If $\psi^{\prime\prime}(G)>\psi^{\prime\prime}(Q_8)=\frac{27}{64}$, then $G$ is abelian.

\item[$(c).$] If $\psi^{\prime\prime}(G)>\psi^{\prime\prime}(S_3)=\frac{13}{36}$, then $G$ is nilpotent.

\item[$(d).$] If $\psi^{\prime\prime}(G)>\psi^{\prime\prime}(A_4)=\frac{31}{144}$, then $G$ is supersolvable.

\item[$(e).$] If $\psi^{\prime\prime}(G)>\psi^{\prime\prime}(A_5)=\frac{211}{3600}$, then $G$ is solvable.
\end{enumerate}
\end{theorem}

We generalize this result by proving the following two theorems in Section \ref{section-main-theorem}.

\begin{theorem}\label{nilpotent-theorem}
Let $G$ be a finite group and $k$ be a positive integer. Then $G$ is nilpotent if 
\begin{itemize}
\item $\psi^{\prime\prime}_k(G)>\psi^{\prime\prime}_k(S_3)=\frac{1+3\cdot 2^k+2\cdot 3^k}{6^{k+1}}$ for $k=1,2$.
\item $\psi^{\prime\prime}_k(G)\ge\frac{1}{2^{k+1}}$ for $k\ge 3$.
\end{itemize}
\end{theorem}

\begin{theorem}\label{main-theorem}
Let $G$ be a finite group and $k$ be a positive integer. Then the following hold.
\begin{enumerate}
\item[$(a).$] If $\psi^{\prime\prime}_k(G)>\psi^{\prime\prime}_k(Q_8)=\frac{1+2^k+6\cdot 4^k}{8^{k+1}}$, then $G$ is abelian.

\item[$(b).$] If $\psi^{\prime\prime}_k(G)>\psi^{\prime\prime}_k(\mathbb{Z}_2\times\mathbb{Z}_2)=\frac{1+3\cdot 2^k}{4^{k+1}}$, then $G$ is cyclic.

\item[$(c).$] If $\psi^{\prime\prime}_k(G)>\psi^{\prime\prime}_k(A_4)=\frac{1+8\cdot 3^k+3\cdot 2^k}{12^{k+1}}$, then $G$ is supersolvable.

\item[$(d).$] If $\psi^{\prime\prime}_k(G)>\psi^{\prime\prime}_k(A_5)=\frac{1+15\cdot 2^k+20\cdot 3^k+24\cdot 5^k}{60^{k+1}}$, then $G$ is solvable.
\end{enumerate}
\end{theorem}

In Theorem \ref{nilpotent-theorem}, the threshold for nilpotency is asymptotically tight, for $k\ge 3$. We show that there exists a sequence of finite non-nilpotent groups $G_n$ for which $\psi''_k(G_n)$ (see Remark \ref{assym-tight}) tends to $1/2^{k+1}$.

\medskip

Let $\mathscr{G}$ denote the class of all finite groups. Besides providing different thresholds for $\psi''$ in \cite{tuarnuauceanu_israel}, as a spin-off, in \cite{lazorec2020density}, the authors considered the function $\psi'':\mathscr{G}\to (0,1]$ defined by $\psi''(G)=\frac{\psi(G)}{|G|^2}$ and studied the density of its image, along with its injectivity and surjectivity. We shall call the image of $\psi''$ as the spectrum of $\psi''$ and denote $\operatorname{Im}\psi''$ as $\operatorname{Spec}\psi''$. They proved the following result.

\begin{theorem}\label{density-psi-double-prime}\cite[Theorem 1.1]{lazorec2020density}
The set 
$$\operatorname{Spec}\psi^{\prime\prime}=\lbrace{\psi^{\prime\prime}(G): G\in\mathscr{G}}\rbrace$$
is dense in $[0,1]$.
\end{theorem}

In a similar manner, for a positive integer $k$, we consider the function $\psi''_k:\mathscr{G}\to (0,1]$ defined by $\psi''_k(G)=\frac{\psi_k(G)}{|G|^{k+1}}$. In this regard, we generalize their result by proving the following theorem.

\begin{theorem}\label{density-theorem}
For any positive integer $k$, the set 
$$\operatorname{Spec}\psi^{\prime\prime}_k=\lbrace{\psi^{\prime\prime}_k(G): G\in\mathscr{G}}\rbrace$$
is dense in $[0,1]$.
\end{theorem}

The proof of Theorem \ref{density-theorem} in Section \ref{section-density} also indicates that the set $\operatorname{Spec}\psi''_k\big|_\mathscr{C}$ is dense in [0, 1], where $\mathscr{C}$ is the class of all finite cyclic groups. Finally we conclude with the comparison among the performances of $\psi',\psi''$ and $\psi''_k$ as indicators in Section \ref{comparison-section}.



\section{Group properties derived from $\psi^{\prime\prime}_k$}\label{section-main-theorem}

In this section we prove Theorems \ref{nilpotent-theorem} and \ref{main-theorem}. To begin with we recall and prove several lemmas.

\begin{lemma}\label{product-quotient}
Let $G$ and $H$ be finite groups and $N\lhd G$. Then for any positive integer $k$ the following are true:
\begin{enumerate}
\item[$(1)$] $\psi^{\prime\prime}_k(G)\leq\psi^{\prime\prime}_k(G/N)$.
\item[$(2)$] $\psi^{\prime\prime}_k(G\times H) \le\psi^{\prime\prime}_k(G)\psi^{\prime\prime}_k(H)$. Equality holds if and only if $\gcd(|G|,|H|)=1$.
\item[$(3)$] $\psi''_k(G) \leq \frac{exp(G)^k}{|G|^k}$.
\end{enumerate}
\end{lemma}
\pf The proofs of $(1)$  and $(2)$ are identical to those of \cite[Proposition 2.6]{herzog2018two} and \cite[Lemma 2.1]{amiri2011sum}, respectively. For $(3)$, the upper bound follows from the fact that 
$$\psi_k(G)=\sum_{g \in G}o(g)^k\le\sum_{g \in G}exp(G)^k=|G|\cdot exp(G)^k.$$ 
Dividing both sides by $|G|^{k+1}$, we have the result.\qed

\begin{lemma}\label{cyclic-p-group-lemma}\cite[Lemma 2.3]{dey2024exact}
Let $G$ be a cyclic group of order $p^n$ for some prime $p$. Then for any positive integer $k$
$$\psi_k(G)=\frac{p^{nk+n+k}(p-1)+p^k-1}{p^{k+1}-1}.$$
\end{lemma}

\begin{lemma}\label{semi-direct-psi-k-lemma}\cite[Remark 2.1]{dey2024exact}
Let $G=P\rtimes H$, where $P$ is a cyclic $p$-group, $|H|>1$ and $(p,|H|)=1$. Then for any positive integer $k$,
$$\psi_k(G)=|P|\psi_k(H)+(\psi_k(P)-|P|)\psi_k(C_H(P)).$$
\end{lemma}

\medskip

A subset $D$ of a real or complex vector space $V$ is convex if for all $x,y\in D$ and $0\le t\le 1$, $(1-t)x+ty\in D$. A real-valued function $f$ defined on a convex subset $D$ of a vector space $V$ is convex if the set 
$$U_f=\lbrace{(x,\lambda): x\in D, \lambda\ge f(x)}\rbrace\subseteq V\times\mathbb{R}$$
of points on and above the graph of $f$ is convex. That is, if $x,y\in D$, and $t\in [0,1]$, then
$$f(tx+(1-t)y)\le tf(x)+(1-t)f(y).$$
For a convex set $D\subseteq\mathbb{R}$, if a function $f:\mathbb{R}\to\mathbb{R}$ satisfies $f^{\prime\prime}(x)\ge 0$ for all $x\in D$, then $f$ is convex on $D$. The following inequality is known as \textit{Jensen's inequality} or \textit{Power mean inequality}.


\begin{lemma}\cite[Proposition 4.1.2]{garling2010inequalities}\label{Jensen-inequality}
Let $f:\mathbb{R}\to\mathbb{R}$ be a convex function on a convex subset $D$ of $\mathbb{R}$ and let $a_1,\ldots,a_n\in\mathbb{R}_{>0}$ be such that $\sum_{i=1}^na_i=1$. Then for any $x_1,\ldots,x_n\in D$,
$$f\left(\sum_{i=1}^na_ix_i\right)\le\sum_{i=1}^na_if(x_i).$$
\end{lemma}

\begin{lemma}\label{threshold-inequality-lemma}
Let $k$ be a positive integer. Then the following hold.
\begin{enumerate}
\item[$(a).$] $\psi^{\prime\prime}_k(\mathbb{Z}_2\times\mathbb{Z}_2)\ge\left(\frac{7}{16}\right)^k$.
\item[$(b).$] $\psi^{\prime\prime}_k(S_3)\ge\left(\frac{13}{36}\right)^k$.
\item[$(c).$] $\psi^{\prime\prime}_k(Q_8)\ge\left(\frac{27}{64}\right)^k$.
\item[$(d).$] $\psi^{\prime\prime}_k(A_4)\ge\left(\frac{31}{144}\right)^k$.
\item[$(e).$] $\psi^{\prime\prime}_k(A_5)\ge\left(\frac{211}{3600}\right)^k$.
\end{enumerate}
\end{lemma}
\pf We only prove part $(e)$. The proofs of the rest follow similarly. First observe that the inequality in $(e)$ is equivalent to showing
$$\frac{60^{k+1}}{1+15\cdot 2^k+20\cdot 3^k+24\cdot 5^k}\le\frac{3600^k}{211^k},$$
i.e.,
$$\frac{60}{1+15\cdot 2^k+20\cdot 3^k+24\cdot 5^k}\le\frac{60^k}{211^k},$$
i.e.,
$$60\cdot 211^k\le 60^k+15\cdot 120^k+20\cdot 180^k+24\cdot 300^k,$$
i.e.,
$$211^k\le\frac{1}{60}\cdot 60^k+\frac{15}{60}\cdot 120^k+\frac{20}{60}\cdot 180^k+\frac{24}{60}\cdot 300^k.$$
Note that for a positive integer $k$, the function $f:\mathbb{R}\to\mathbb{R}$ defined as $f(x)=x^k$ is a convex function on $(0,\infty)$, since $f^{\prime\prime}(x)\ge 0$ for $x>0$. Let $a_1=\frac{1}{60},a_2=\frac{15}{60},a_3=\frac{20}{60},a_4=\frac{24}{60}$ and let $x_1=60,x_2=120,x_3=180,x_4=300$. Then from Lemma \ref{Jensen-inequality}, we have
$$f\left(\frac{1}{60}\cdot 60+\frac{15}{60}\cdot 120+\frac{20}{60}\cdot 180+\frac{24}{60}\cdot 300\right)\le \frac{1}{60}f(60)+\frac{15}{60}f(120)+\frac{20}{60}f(180)+\frac{24}{60}f(300),$$
i.e.,
$$211^k\le\frac{1}{60}\cdot 60^k+\frac{15}{60}\cdot 120^k+\frac{20}{60}\cdot 180^k+\frac{24}{60}\cdot 300^k.$$
\qed

\subsection{Proof of Theorem \ref{nilpotent-theorem}}

\medskip

To prove the threshold for nilpotent groups we prove the following lemmas first.

\begin{lemma}\label{reducing-quotient}
Let $G$ be a finite group such that $G\cong A \rtimes H$ and $M$ be a subgroup of $H$, which is normal in $G$. Then
$$G/M \cong A \rtimes (H/M).$$
Moreover, if $G$ is nontrivial semidirect product, then $G/M$ is as well and vice versa.
\end{lemma}
\pf  Consider the natural homomorphism $\tau : G \rightarrow G/M=G_1$ (say). Then $\tau(A)=AM/M \cong A/A\cap M \cong A \leq G_1$ (since $A \cap H$ is trivial) and $\tau(H)=HM/M \cong H/M \leq G_1$ (since $M\leq H$). Now since $\tau$ is an onto homomorphism and $G=AH$ and $A$ is normal in $G$, then $\tau(A)$ is normal in $G_1$ and ${G_1}={\tau(A)}{\tau(H)}$. Let $gM \in \tau(A) \cap \tau(H)$. Then there exists $x \in A$ and $y \in H$ such that $\tau(x)=gM=\tau(y)$ i.e. $x{y^{-1}} \in M \leq H$ and hence, $x \in H$. Thus $x=e$ and consequently $gM=M$. So $G/M \cong \tau(A) \rtimes \tau(H)$ i.e.,
$$G/M \cong A \rtimes (H/M).$$
The second part of this proof is obvious. \qed 

\begin{lemma} \label{psi''_k for semi-direct}
Let $G=\mathbb{Z}_{p^n}\rtimes\mathbb{Z}_2$ be a nontrivial semi-direct product, where $p$ is an odd prime and $n$ is a positive integer. Then
\begin{itemize}
\item $\psi^{\prime\prime}_k(G)\le\frac{1+3\cdot 2^k+2\cdot 3^k}{6^{k+1}}$ for $k=1,2$.
\item $\psi^{\prime\prime}_k(G)<\frac{1}{2^{k+1}}$ for $k\ge 3$.
\end{itemize}
\end{lemma}
\pf By Lemma \ref{semi-direct-psi-k-lemma}, we have
$$\psi_k(G)=\psi_k(\mathbb{Z}_{p^n}\rtimes\mathbb{Z}_2)=p^n\psi_k(\mathbb{Z}_2)+(\psi_k(\mathbb{Z}_{p^n})-p^n)\psi_k(C_{\mathbb{Z}_2}(\mathbb{Z}_{p^n})).$$
Since the semi-direct product is non-trivial, we have $C_{\mathbb{Z}_2}(\mathbb{Z}_{p^n})=\lbrace{1}\rbrace$ and so
$$\psi_k(G)=(1+2^k)p^n+\frac{p^{nk+n+k}(p-1)+p^k-1}{p^{k+1}-1}-p^n,$$
i.e.,
$$\psi^{\prime\prime}_k(G)=\frac{p^{nk+n+k+1}+2^kp^{n+k+1}+p^k-p^{nk+n+k}-2^kp^n-1}{2^{k+1}(p^{nk+n+k+1}-p^{nk+n})}.$$
The proof of $\psi^{\prime\prime}_k(G)\le\frac{1+3\cdot 2^k+2\cdot 3^k}{6^{k+1}}$ for $k=1,2$ goes identical with that of \cite[Lemma 2.3]{tuarnuauceanu_israel}. The equality $\psi^{\prime\prime}_k(G)=\frac{1+3\cdot 2^k+2\cdot 3^k}{6^{k+1}}$ for $k=1,2$ holds if and only if $p=3$ and $n=1$, i.e., $G\cong S_3$. To show $\psi^{\prime\prime}_k(G)<\frac{1}{2^{k+1}}$ for $k\ge 3$ is equivalent to showing 
$$p^{nk+n+k+1}+2^k p^{n+k+1}+p^k-p^{nk+n+k}-2^k p^n-1<p^{nk+n+k+1}-p^{nk+n},$$
i.e.,
$$2^k p^{n+k+1}+p^k-2^k p^n- 1<p^{nk+n+k}-p^{nk+n},$$
i.e.,
$$2^k p^n(p^{k+1}-1)<(p^{nk+n}- 1)(p^k-1),$$
i.e.,
$$2^k\cdot\frac{p^{k+1}-1}{p^k-1}<p^{nk}-\frac{1}{p^n}.$$
Since $n\ge 1$, $p^{nk}\ge p^k$ and $p^n\ge p$. Hence $p^{nk}-\frac{1}{p^n}\ge p^k-\frac{1}{p}$. We show that 
$$2^k\cdot\frac{p^{k+1}-1}{p^k-1}<p^k-\frac{1}{p},$$
which is equivalent to showing
$$p^k-2^kp-1>0,$$
which can be verified by induction for all $p\ge 3$ and $k\ge 3$. Hence we have $\psi^{\prime\prime}_k(G)<\frac{1}{2^{k+1}}$ for $k\ge 3$ as desired.\qed

\medskip

\noindent\textbf{Proof of Theorem \ref{nilpotent-theorem}:}
Let, if possible, $G$ be a counterexample of minimum order, i.e., $G$ is a non-nilpotent group of minimum order such that
\begin{itemize}
\item $\psi_k''(G)>\frac{1+3\cdot 2^k+2\cdot 3^k}{6^{k+1}}$ for $k=1,2$.
\item $\psi_k''(G)\ge \frac{1}{2^{k+1}}$ for $k\ge 3$.
\end{itemize}
Since 
$$\frac{1}{|G|}\sum_{g \in G} o(g)^k>\left(\frac{1+3\cdot 2^k+2\cdot 3^k}{6^{k+1}}\right)|G|^k\quad \text{for} \quad k=1,2$$
and
$$\frac{1}{|G|}\sum_{g \in G} o(g)^k\ge\left(\frac{1}{2^{k+1}}\right)|G|^k \quad \text{for} \quad k\ge 3,$$
there exists $x \in G$ such that
$$o(x)>\left(\frac{1+3\cdot 2^k+2\cdot 3^k}{6^{k+1}}\right)^{1/k}|G|\quad \text{for} \quad k=1,2$$
and
$$o(x)\ge\left(\frac{1}{2^{k+1}}\right)^{1/k}|G|\quad \text{for} \quad k\ge 3.$$
Let $H=\langle{x}\rangle$. Then we have for $k=1,2$,
$$[G:H]<\left(\frac{6^{k+1}}{1+3\cdot 2^k+2\cdot 3^k}\right)^{1/k}\le\frac{36}{13}<3,$$
where the second inequality follows from Lemma \ref{threshold-inequality-lemma} (b). And for $k\ge 3$,
$$[G:H]\le\left({2^{k+1}} \right)^{1/k}<3.$$
Hence we get
$$[G:H]\le 2.$$
Since $G$ is non-nilpotent, $[G:H]=2$, i.e., $H$ is normal in $G$ and also $\pi(G)>1$. By Schur-Zassenhaus theorem, $G\cong \mathbb{Z}_t\rtimes M$, where $M$ is a Sylow $2$-subgroup of $G$ of order $2^n$. Note that $\mathbb{Z}_t$ is the internal direct product of all the cyclic normal Sylow subgroups (mentioned above) of $G$. As $G$ is non-nilpotent, at least one of those cyclic normal Sylow subgroups, say  Sylow $p_1$-subgroup $P_1$, does not commute with $M$. Now, if $\mathbb{Z}_t$ contains another cyclic normal Sylow $p_2$-subgroup, say $P_2$, then $G/P_2\cong \mathbb{Z}_u \rtimes M$, where $u=t/|P_2|$ and $G/P_2$ is non-nilpotent. By Lemma \ref{product-quotient} (1), we have $\psi_k''(G/P_2)\geq \psi_k''(G) $ which leads to a contradiction to minimality of $G$ as a counterexample. Hence without loss of generality, $$G\cong \mathbb{Z}_{p^m}\rtimes M,$$ 
where the semidirect product is non-trivial. Let $K\leq H$ such that $|K|=2^{n-1}$. As $K$ is a characteristic subgroup of $H$ and $H$ is normal in $G$, $K$ is normal in $G$. Thus, by Lemma \ref{reducing-quotient}, we get a non-trivial semidirect product 
$$G/K\cong \mathbb{Z}_{p^m}\rtimes M/K.$$
By the minimality of $G$, we have $|K|=1$ and hence
$$G\cong \mathbb{Z}_{p^m}\rtimes \mathbb{Z}_2,$$
which contradicts the Lemma \ref{psi''_k for semi-direct}.\qed

\begin{remark}\label{assym-tight}
For the group $G=D_p$, where $p$ is a prime, we have
$$\psi''_k(D_p)=\frac{p^{k+1}-p^k+2^kp+1}{(2p)^{k+1}}.$$
Thus
$$\lim_{p\to\infty}\psi''_k(D_p)=\frac{1}{2^{k+1}}.$$
This shows that the threshold for nilpotency given by Theorem \ref{nilpotent-theorem} is asymptotically tight, for $k\ge 3$.
\end{remark}

\subsection{Proof of Theorem \ref{main-theorem}}

\medskip

\begin{theorem}\label{lucchini-theorem}\cite[Theorem 2.20]{isaacs2008_book}
Let $A$ be a cyclic proper subgroup of a finite group $G$, and let $K = \operatorname{Core}_G(A)$, then $[A:K]<[G:A]$, and in particular, if $|A|\geq [G:A]$, then $K$ is non-trivial.
\end{theorem}

\begin{theorem}\label{prime-power-index-solvable-theorem}\cite[Theorem 13.10.1]{scott1987group}
Suppose that a finite group $G$ contains a subgroup $B$ of prime
power index and $B$ contains a cyclic subgroup $H$ of index $[B:H]\le 2$. Then $G$ is solvable.
\end{theorem}

\begin{lemma}\label{Huppert's-classification-lemma}\cite[Theorem 14.9(b)]{huppert2025_book}
Let $G$ be a non abelian $2$-group such that $G$ has a cyclic subgroup of index $2$. Then $G$ is isomorphic to one of the following.
$$M_{2^n}=\langle{x,y\mid x^{2n-1}=y^2=1, yxy=x^{2^{n-2}+1}}\rangle, \quad n\ge 4;$$
$$D_{2^n}=\langle{x,y\mid x^{2n-1}=y^2=1, yxy=x^{-1}}\rangle, \quad n\ge 2;$$
$$Q_{2^n}=\langle{x,y\mid x^{2n-1}=y^4=1, yxy^{-1}=x^{2^{n-1}-1}}\rangle, \quad n\ge 3;$$
$$SD_{2^n}=\langle{x,y\mid x^{2n-1}=y^2=1, yxy=x^{2^{n-2}-1}}\rangle, \quad n\ge 4.$$
\end{lemma}

The following result is straightforward and hence stated without a proof.

\begin{lemma}\label{non-cyclic-abelian-$2$-group-lemma}
Let $G$ be a non-cyclic abelian $2$-group such that $G$ has a cyclic subgroup of index $2$. Then $G\cong\mathbb{Z}_{2^{n-1}}\times\mathbb{Z}_2$.
\end{lemma}

\begin{table}[h]
    \centering
    \begin{tabular}{|c|c|c|}
        \hline
        $G$ & $|G|$ & \textbf{$\psi''_k(G)$} \rule[-2.5ex]{0pt}{6ex} \\ 
        \hline
        $\mathbb{Z}_{2^{n-1}} \times \mathbb{Z}_2$ & 
        $2^n(n \geq 2)$ & 
        $\displaystyle \frac{2^{n(k+1)} + 2^{2k+1} - 2^k - 1}{2^{n(k+1)}(2^{k+1} - 1)}$ \rule[-4.5ex]{0pt}{10ex} \\
        \hline
        $M_{2^n}$ & 
        $2^n(n \geq 4)$ & 
        $\displaystyle \frac{2^{n(k+1)} + 2^{2k+1} - 2^k - 1}{2^{n(k+1)}(2^{k+1} - 1)}$ \rule[-4.5ex]{0pt}{10ex} \\
        \hline
        $D_{2^n}$ & 
        $2^{n+1}(n \geq 2)$ & 
        $\displaystyle \frac{2^{n(k+1)+k} + 2^{n+2k+1} - 2^{n+k} + 2^k - 1}{2^{(n+1)(k+1)}(2^{k+1} - 1)}$ \rule[-4.5ex]{0pt}{10ex} \\
        \hline
        $Q_{2^n}$ & 
        $2^n(n \geq 3)$ & 
        $\displaystyle \frac{2^{n(k+1)-1} + 2^{n+3k} - 2^{n+2k-1} + 2^k - 1}{2^{n(k+1)}(2^{k+1} - 1)}$ \rule[-4.5ex]{0pt}{10ex} \\
        \hline
        $SD_{2^n}$ & 
        $2^n(n \geq 4)$ & 
        $\displaystyle \frac{2^{n(k+1)-1} + 2^{n+3k-1} + 2^{n+2k-2} - 2^{n+k-2} + 2^k - 1}{2^{n(k+1)}(2^{k+1} - 1)}$ \rule[-4.5ex]{0pt}{10ex} \\
        \hline
    \end{tabular}
    \caption{$\psi''_k$ values of some groups}
    \label{Huppert's Table}
\end{table}

The $\psi''_k$ values of the groups mentioned in Lemma \ref{Huppert's-classification-lemma} and Lemma \ref{non-cyclic-abelian-$2$-group-lemma} are given in Table \ref{Huppert's Table}.

\medskip

\noindent\textbf{Proof of Theorem \ref{main-theorem}:}
\begin{enumerate}
\item[$(a).$]
First observe that, 
$$\frac{1+2^k+6\cdot 4^k}{8^{k+1}}>\frac{1+3\cdot 2^k+2\cdot 3^k}{6^{k+1}} \quad \text{for} \quad k=1,2$$ and
$$\frac{1+2^k+6\cdot 4^k}{8^{k+1}}>\frac{1}{2^{k+1}} \quad \text{for} \quad k \geq 3.$$
Hence by Theorem \ref{nilpotent-theorem}, $\psi''_k(G)>\frac{1+2^k+6\cdot 4^k}{8^{k+1}}$ implies $G$ is nilpotent, for all $k\ge 1$. Let, if possible, $G$ be a counterexample of minimum order, i.e., $G$ is a non-abelian group of minimum order such that 
$$\psi''_k(G)>\frac{1+2^k+6\cdot 4^k}{8^{k+1}}.$$
Since 
$$\frac{1}{|G|}\sum_{g \in G}o(g)^k>\left(\frac{1+2^k+6\cdot 4^k}{8^{k+1}}\right)|G|^k,$$ 
there exists $x\in G$ such that $o(x)>\left(\frac{1+2^k+6\cdot 4^k}{8^{k+1}}\right)^{1/k}|G|$. Let $H=\langle{x}\rangle$. Then 
$$[G:H]< \left(\frac{8^{k+1}}{1+2^k+6\cdot 4^k}\right)^{1/k}\le\frac{64}{27}<3,$$
where the second inequality follows from Lemma \ref{threshold-inequality-lemma} (c). Hence $[G:H]=2$ and therefore all Sylow $p$-subgroups of $G$ (if exists) are cyclic and normal in $G$, $p \neq 2$ and since $G$ is nilpotent then $G \cong \mathbb{Z}_t \times Q$, where $Q$ is the Sylow $2$-subgroup of $G$ and $t \geq 1$. So $Q$ has to be non-abelian. Then by Lemma \ref{product-quotient} (1) we have 
$$\psi''_k(Q) = \psi''_k(G/{\mathbb{Z}_t})\geq \psi''_k(G)>\frac{1+2^k+6\cdot 4^k}{8^{k+1}}, \quad \text{for all $k\ge 1$}$$
i.e.,
\begin{equation} \label{abelian-quotient-eqn}
\psi''_k(Q)>\frac{1+2^k+6\cdot 4^k}{8^{k+1}}, \quad \text{for all $k\ge 1$}.
\end{equation}
Since $Q$ is non-abelian and $[G:H]=2$, then $Q$ contains a cyclic subgroup of index $2$. Hence $Q$ is isomorphic to one of the groups mentioned in Lemma \ref{Huppert's-classification-lemma}. From Table \ref{Huppert's Table} and Lemma A in Appendix, it can be seen that 
$$\psi''_k(Q)\le \frac{1+2^k+6\cdot 4^k}{8^{k+1}}, \quad \text{for all $k\ge 1$},$$
which contradicts \eqref{abelian-quotient-eqn}.\qed

\item[$(b).$]
Observe that, $$\frac{1+3\cdot 2^k}{4^{k+1}}>\frac{1+2^k+6\cdot 4^k}{8^{k+1}}, \quad \text{for all $k\ge 1$}.$$
So by (a), $\psi''_k(G)>\frac{1+3\cdot 2^k}{4^{k+1}}$ implies $G$ is abelian, for all $k\ge 1$. Let, if possible, $G$ be a counterexample of minimum order, i.e., $G$ is a non-cyclic group of minimum order such that 
$$\psi''_k(G)>\frac{1+3\cdot 2^k}{4^{k+1}}.$$
Since 
$$\frac{1}{|G|}\sum_{g \in G}o(g)^k>\left(\frac{1+3\cdot 2^k}{4^{k+1}}\right)|G|^k,$$ 
there exists $x\in G$ such that $o(x)>\left(\frac{1+3\cdot 2^k}{4^{k+1}}\right)^{1/k}|G|$. Let $H=\langle{x}\rangle$. Then 
$$[G:H]< \left(\frac{4^{k+1}}{1+3\cdot 2^k}\right)^{1/k}\le\frac{16}{7}<3,$$
where the second inequality follows from Lemma \ref{threshold-inequality-lemma} (a). Hence $[G:H]=2$ and therefore for $p\neq 2$, all Sylow $p$-subgroups of $G$ (if exists) are cyclic and normal in $G$. Since $G$ is abelian, $G \cong \mathbb{Z}_t \times Q$, where $Q$ is the abelian Sylow $2$-subgroup of $G$ and $t \geq 1$. Then similarly by the previous part $Q$ has to be non-cyclic together with a cyclic subgroup of index $2$ and
\begin{equation} \label{cyclic_quotient_eqn}
\psi''_k(Q)>\frac{1+3\cdot 2^k}{4^{k+1}}, \quad \text{for all $k\ge 1$}.
\end{equation}

Since $Q$ is abelian and non-cyclic, by Lemma \ref{non-cyclic-abelian-$2$-group-lemma}, we have $Q \cong \mathbb{Z}_{2^{n-1}} \times \mathbb{Z}_2$, where $n \geq 2$. From Table \ref{Huppert's Table} and Lemma B in Appendix, it can be seen that
$$\psi''_k(Q)\le \frac{1+3\cdot 2^k}{4^{k+1}}, \quad \text{for all $k\ge 1$},$$
which contradicts \eqref{cyclic_quotient_eqn}.\qed

\item[$(c).$] Let, if possible, $G$ be a counterexample of minimum order, i.e., $G$ is a non-supersolvable group of minimum order such that 
\begin{equation}\label{A4-threshold-eqn}
\psi^{\prime\prime}_k(G)>\frac{1+8\cdot3^k+3\cdot 2^k}{12^{k+1}}.
\end{equation}
Since 
$$\frac{1}{|G|}\sum_{g \in G}o(g)^k>\left(\frac{1+8\cdot 3^k+3\cdot 2^k}{12^{k+1}}\right)|G|^k,$$ 
there exists $x\in G$ such that $o(x)>\left(\frac{1+8\cdot 3^k+3\cdot 2^k}{12^{k+1}}\right)^{1/k}|G|$. Let $H=\langle{x}\rangle$. Then 
$$[G:H]<\left(\frac{12^{k+1}}{1+8\cdot 3^k+3\cdot 2^k}\right)^{1/k}\le\frac{144}{31}<5,$$
where the second inequality follows from Lemma \ref{threshold-inequality-lemma} (d).

Hence $[G:H]\le 4$. Let $N=\operatorname{Core}_G(H)$. Then by Theorem \ref{lucchini-theorem},
\begin{equation}\label{index-eqn-supersolvable}
[G:N]=[G:H][H:N]\le 4\cdot 3=12.
\end{equation}
If $N\neq\lbrace{1}\rbrace$, then by Lemma \ref{product-quotient} ($1$), 
$$\psi^{\prime\prime}_k(G/N)>\frac{1+8\cdot 3^k+3\cdot 2^k}{12^{k+1}},$$
which by the minimality of $G$ yields that $G/N$ is supersolvable and so is $G$, contrary to our assumption.

If $N=\lbrace{1}\rbrace$, then by (\ref{index-eqn-supersolvable}), $|G|\le 12$. Since $G$ is non supersolvable, $G\cong A_4$, contradicting (\ref{A4-threshold-eqn}). All these cases together imply that no such counterexample exists for $G$.

\item[$(d).$] We again start with $G$, a counterexample of minimum order. Then $G$ is a non solvable group of minimum order such that
\begin{equation}\label{A5-threshold-eqn}
\psi^{\prime\prime}_k(G)>\frac{1+15\cdot 2^k+20\cdot 3^k+24\cdot 5^k}{60^{k+1}}.
\end{equation}

Proceeding as above, we have an element $x\in G$ such that 
$$[G:H]<\frac{60^{k+1}}{1+15\cdot 2^k+20\cdot 3^k+24\cdot 5^k}\le\frac{3600}{211}<18,$$
where $H=\langle{x}\rangle$. The second inequality follows from Lemma \ref{threshold-inequality-lemma} (e).

Hence $[G:H]\le 17$. If $[G:H]\in\lbrace{16,17}\rbrace$, then by Theorem \ref{prime-power-index-solvable-theorem}, $G$ is solvable, contrary to our assumption. Hence we only work with $[G:H]\le 15$. Let $N=\operatorname{Core}_G(H)$. Then by Theorem \ref{lucchini-theorem},
\begin{equation}\label{index-eqn-solvable}
[G:N]=[G:H][H:N]\le 15\cdot 14=210.
\end{equation}
If $N\neq\lbrace{1}\rbrace$, then by Lemma \ref{product-quotient} ($1$), 
$$\psi^{\prime\prime}_k(G/N)>\frac{1+15\cdot 2^k+20\cdot 3^k+24\cdot 5^k}{60^{k+1}},$$
which by the minimality of $G$ yields that $G/N$ is solvable and so is $G$, contrary to our assumption.

If $N=\lbrace{1}\rbrace$, then by (\ref{index-eqn-solvable}), $|G|\le 210$. Since $G$ is non solvable, one of the following holds:
\begin{enumerate}[label=(\roman*)]
\item $|G|=60$ and $G\cong A_5$.
\item $|G|=120$ and $G\cong A_5\times\mathbb{Z}_2$ or $G\cong S_5$ or $G\cong\operatorname{SL}(2,5)$.
\item $|G|=168$ and $G\cong\operatorname{PSL}(2,7)$.
\item $|G|=180$ and $G\cong A_5\times\mathbb{Z}_3$.
\end{enumerate}
Clearly, $G\cong A_5$ contradicts (\ref{A5-threshold-eqn}). If $G\cong A_5\times\mathbb{Z}_n$ with $n=2,3$, then Lemma \ref{product-quotient} ($2)$, leads to
$$\psi^{\prime\prime}_k(G)\le\psi^{\prime\prime}_k(A_5)\psi^{\prime\prime}_k(\mathbb{Z}_n)<\psi^{\prime\prime}_k(A_5)=\frac{1+15\cdot 2^k+20\cdot 3^k+24\cdot 5^k}{60^{k+1}},$$
contradicting (\ref{A5-threshold-eqn}).

For $G\cong S_5,\operatorname{SL}(2,5)$ and $\operatorname{PSL}(2,7)$, we have (see Lemma C, Appendix)
$$\psi^{\prime\prime}_k(G)<\frac{1+15\cdot 2^k+20\cdot 3^k+24\cdot 5^k}{60^{k+1}}=\psi^{\prime\prime}_k(A_5),$$
contradicting (\ref{A5-threshold-eqn}). Hence no such counterexample exists for $G$.
\end{enumerate}

\section{Spectrum of $\psi^{\prime\prime}_k$}\label{section-density}

In this section we consider the spectrum of $\psi^{\prime\prime}_k$. For a finite group $G$, we denote the order set of $G$ by $\omega(G)$. For a positive integer $n$, define $[n]\coloneqq\lbrace{1,2,\ldots,n}\rbrace$.

\begin{lemma}\label{Psi-k-G-odd-lemma}
Let $G$ be a finite group and $k$ be a positive integer. Then $\psi_k(G)$ is odd.
\end{lemma}
\pf We have $$\psi_k(G)=\sum_{d\in\omega(G)}c_dd^k=1+\sum_{d\in\omega(G),d\ge 2}c_dd^k.$$ 
Since $\varphi(d)\mid c_d$, $c_d$ is even for $d>2$ and hence $c_dd^k$ is even for $d>2$. Thus we have the result.\qed

\begin{proposition}
For any positive integer $k$, the function $\psi^{\prime\prime}_k$ is neither surjective nor injective
\end{proposition}
\pf We choose $a,b\in\mathbb{N}$ such that $a$ is even, $b$ is odd and $a<b$. We show that there does not exist $G\in\mathscr{G}$ such that $\psi^{\prime\prime}_k(G)=\frac{a}{b}$. If possible, let $G$ be such a group. Then we have $b\psi_k(G)=a|G|^{k+1}$. Since $b$ is odd, the left-hand side is odd by Lemma \ref{Psi-k-G-odd-lemma} and we get a contradiction.

To show $\psi''_k$ is not injective, consider the Heisenberg group $G=(\mathbb{Z}_3\times\mathbb{Z}_3)\rtimes\mathbb{Z}_3$ and the elementary abelian group $H=\mathbb{Z}_3\times\mathbb{Z}_3\times\mathbb{Z}_3$, both of order $27$. Then for any positive integer $k$, one can easily see that
$$\psi^{\prime\prime}_k(G)=\psi^{\prime\prime}_k(H)=\frac{1+26\cdot 3^k}{27^{k+1}}.$$\qed



\begin{remark}
If two non-isomorphic groups have identical order sequences, it can be easily seen that their $\psi^{\prime\prime}_k$ values will be the same.
\end{remark}

To prove the density of $\operatorname{Spec}\psi''_k$, we recall and prove a few lemmas. 


\begin{lemma}\label{seq-density-lemma}\cite[Proposition on pp. 863]{nitecki2015cantorvals}
Let $(x_n)_{n\ge 1}$ be a sequence of positive real numbers such that $\lim_{n\to\infty}x_n=0$ and $\sum_{n=1}^{\infty}x_n$ is divergent. Then the set containing the sums of all finite subsequences of $(x_n)_{n\ge 1}$ is dense in $[0,\infty)$.
\end{lemma}

\begin{lemma}\label{two-seq-same-nature-lemma}\cite[Theorem 10.9]{apostol1967one}
Let $(a_n)_{n\ge 1}$ and $(b_n)_{n\ge 1}$ be two sequences of positive real numbers such that $\lim_{n\to\infty}\frac{a_n}{b_n}=\alpha\in [0,\infty]$. If $\alpha\in (0,\infty)$, then the series $\sum_{n=1}^{\infty}a_n$ and $\sum_{n=1}^{\infty}b_n$ have the same nature.
\end{lemma}

Let $\lbrace{u_n}\rbrace$ be a sequence of real numbers. The number $p_n\coloneqq\prod_{k=1}^nu_k$ is called the $n$-th partial product of the infinite product $\prod_{n=1}^{\infty}u_n$ and $u_n$ is called the $n$-th factor of the product. Given an infinite product $\prod_{n=1}^{\infty}u_n$, if no factor $u_n$ is zero, we say the product converges if there exists a number $p\neq 0$ such that $\lbrace{p_n}\rbrace$ converges to $p$. In this case, $p$ is called the value of the product and we write $p=\prod_{n=1}^{\infty}u_n$. If $\lbrace{p_n}\rbrace$ converges to zero, we say the product diverges to zero.

\begin{lemma}\cite[Theorem 8.55]{Apostol1974}\label{product-convergence-lemma}
Let $(a_n)_{n\ge 1}$ be a sequence of non-negative real numbers. Then the product $\prod_{n=1}^{\infty}(1-a_n)$ converges if and only if $\sum_{n=1}^{\infty}a_n$ converges.
\end{lemma}

Lemma \ref{product-convergence-lemma} says that if $\sum_{n=1}^{\infty}a_n$ converges, then $\prod_{n=1}^{\infty}(1-a_n)$ converges to a non-zero number. But if $\sum_{n=1}^{\infty}a_n$ diverges, then $\prod_{n=1}^{\infty}(1-a_n)$ diverges to zero.

\begin{lemma}\label{prime-product-limit-zero-lemma}
Let $p_i$ denote the $i$-th prime and let $k$ be a positive integer. Then
$$\prod_{i=1}^{\infty}\left(\frac{p_i^k(p_i-1)}{p_i^{k+1}-1}\right)=0$$
\end{lemma}
\pf Consider the sequence $(u_n)_{n\ge 1}$, where
$$u_n\coloneqq\prod_{i=1}^{n}\left(\frac{p_i^k(p_i-1)}{p_i^{k+1}-1}\right)=\prod_{i=1}^{n}\left(\frac{p_i^k}{1+p_i+p_i^2+\cdots+p_i^k}\right).$$
Let $b_i\coloneqq\frac{1+p_i+\cdots+p_i^{k-1}}{1+p_i+\cdots+p_i^k}$. Since $p_i\ge 2$ and $k\ge 1$, we have
$$b_i=\frac{1}{p_i}\left(\frac{p_i+p_i^2+\cdots+p_i^k}{1+p_i+\cdots+p_i^k}\right)=\frac{1}{p_i}\left(1-\frac{1}{1+p_i+\cdots+p_i^k}\right)\ge \frac{1}{p_i}\left(1-\frac{1}{3}\right)=\frac{2}{3p_i}.$$
Since $\sum_{i=1}^{\infty}\frac{1}{p_i}$ is divergent, $\sum_{i=1}^{\infty}b_i$ is divergent. Hence by Lemma \ref{product-convergence-lemma}, $\prod_{i=1}^{\infty}(1-b_i)$ diverges to $0$, i.e.,
$$\prod_{i=1}^{\infty}\left(\frac{p_i^k}{1+p_i+p_i^2+\cdots+p_i^k}\right)=0.$$
\qed

\medskip

\noindent{\bf Proof of Theorem \ref{density-theorem}:} Let $I$ be a finite subset of $\mathbb{N}$ and $p_i$ denote the $i$-th prime. Consider the sequence $(T_{n}(I))_{n\ge 1}\subset\operatorname{Spec}\psi^{\prime\prime}_k\big|_{\mathscr{C}}$, where
$$T_{n}(I)\coloneqq\psi^{\prime\prime}_k\big|_{\mathscr{C}}\left(\prod_{i\in I}\mathbb{Z}_{p_i^n}\right)=\psi^{\prime\prime}_k\left(\prod_{i\in I}\mathbb{Z}_{p_i^n}\right).$$
Then by Lemma \ref{cyclic-p-group-lemma} and Lemma \ref{product-quotient} ($2$),
$$\lim_{n\to\infty}T_{n}(I)=\lim_{n\to\infty}\psi^{\prime\prime}_k\left(\prod_{i\in I}\mathbb{Z}_{p_i^n}\right)=\lim_{n\to\infty}\left(\prod_{i\in I}\frac{p_i^{nk+n+k}(p_i-1)+p_i^k-1}{p_i^{n(k+1)}(p_i^{k+1}-1)}\right)=\prod_{i\in I}\left(\frac{p_i^k(p_i-1)}{p_i^{k+1}-1}\right).$$
Hence we have,
$$\mathcal{A}\subseteq\overline{\operatorname{Spec}\psi^{\prime\prime}_k\big|_{\mathscr{C}}}\subseteq\overline{\operatorname{Spec}\psi^{\prime\prime}_k}\subseteq[0,1],$$
where
$$\mathcal{A}\coloneqq\left\lbrace{\prod_{i\in I}\left(\frac{p_i^k(p_i-1)}{p_i^{k+1}-1}\right):I\subset\mathbb{N},|I|<\infty}\right\rbrace.$$
It suffices to show that $\mathcal{A}$ is dense in $[0,1]$, i.e., $\overline{\mathcal{A}}=[0,1]$.

Consider the sequence $(x_i)_{i\ge 1}\subset(0,\infty)$, where $x_i=\ln\left(\frac{p_i^{k+1}-1}{p_i^k(p_i-1)}\right)$. Hence we have, 
$$\lim_{i\to\infty}\frac{x_i}{1/p_i}=\lim_{i\to\infty}\frac{\ln\left(\frac{p_i^{k+1}-1}{p_i^k(p_i-1)}\right)}{1/p_i}=1.$$
Since $\sum_{i=1}^{\infty}\frac{1}{p_i}$ is divergent, by Lemma \ref{two-seq-same-nature-lemma}, $\sum_{i=1}^{\infty}x_i$ is also divergent. Also we have,
$$\lim_{i\to\infty}x_i=\lim_{i\to\infty}\ln\left(\frac{p_i^{k+1}-1}{p_i^k(p_i-1)}\right)=\ln 1=0.$$
Hence by Lemma \ref{seq-density-lemma}, the set containing the sums of all finite subsequences of $(x_i)_{i\ge 1}$ is dense in $[0,\infty)$. Consequently we have,
$$\overline{\left\lbrace{\sum_{i\in I}x_i:I\subset\mathbb{N}, |I|<\infty}\right\rbrace}=[0,\infty)$$ i.e.,
$$\overline{\left\lbrace{\sum_{i\in I}\ln\left(\frac{p_i^{k+1}-1}{p_i^k(p_i-1)}\right):I\subset\mathbb{N}, |I|<\infty}\right\rbrace}=[0,\infty)$$ i.e.,
$$\overline{\left\lbrace{\ln\left(\prod_{i\in I}\left(\frac{p_i^{k+1}-1}{p_i^k(p_i-1)}\right)\right):I\subset\mathbb{N}, |I|<\infty}\right\rbrace}=[0,\infty).$$
Since the function 
$$\exp:\mathbb{R}\to\mathbb{R}$$
$$x\mapsto e^x,\quad \forall x\in\mathbb{R}$$
is continuous and $[1,\infty)$ is a closed set of $\mathbb{R}$, we have $\overline{\mathcal{B}}=[1,\infty)$, where
$$\mathcal{B}\coloneqq\left\lbrace{\prod_{i\in I}\left(\frac{p_i^{k+1}-1}{p_i^k(p_i-1)}\right):I\subset\mathbb{N}, |I|<\infty}\right\rbrace.$$
Let $a\in (0,1]$. Then $\frac{1}{a}\in [1,\infty)$. Hence $\frac{1}{a}\in\overline{\mathcal{B}}$, i.e., there exists a sequence $(y_t)_{t\ge 1}\subset\mathcal{B}$ such that $\lim_{t\to\infty}y_t=\frac{1}{a}$. Thus $\lim_{t\to\infty}\frac{1}{y_t}=a$. Since $\left(\frac{1}{y_t}\right)_{t\ge 1}\subset\mathcal{A}$, $a\in\overline{\mathcal{A}}$. Thus $(0,1]\subseteq\overline{\mathcal{A}}$.

To show that $0\in\overline{\mathcal{A}}$, we take $I_n=[n]$. Then by Lemma \ref{prime-product-limit-zero-lemma}, we have $\lim_{n\to\infty}u_n=0$, where $u_n\coloneqq\prod_{i\in I_n}\left(\frac{p_i^k(p_i-1)}{p_i^{k+1}-1}\right)\in\mathcal{A}$. This completes the proof.\qed

\section{Comparison with other parameters}\label{comparison-section}
In this section, we compare the performance of the parameters $\psi'(G),\psi''(G)$ with the newly proposed measure $\psi''_k(G)$. We show that there exist groups whose nilpotency and (super)solvability cannot be guaranteed by the known thresholds of $\psi'(G)$ (Theorem \ref{thresholds-for-psi-prime}) and $\psi''(G)$ (Theorem \ref{thresholds-for-psi-double-prime}), but can be done by $\psi''_k(G)$ for some $k$.

\begin{example}
Consider the group $G=\mathbb{Z}_{2}\times \mathbb{Z}_2 \times D_8$ of order $64$. One can compute, either manually or by SageMath \cite{sage}, that $$\psi'(G)=\dfrac{239}{2731}<\dfrac{211}{1617};\quad  \psi''(G)=\dfrac{239}{4096}<\dfrac{211}{3600};$$ Thus, the solvability of $G$ is not guaranteed by these parameters. However, we see that $$\psi''_2(G)=\dfrac{1309}{262144}>\dfrac{841}{216000}.$$ Hence by Theorem \ref{thresholds-for-psi-double-prime}, we can say that $G$ is solvable. Similarly, the solvability of $A_4 \times S_3$ can not be derived from their respective values of $\psi'$ and $\psi''$. However $\psi_5''$ does the job.
\end{example}

\begin{example}
Consider the group $G=\mathbb{Z}_{4}\times \mathbb{Z}_4$ of order $16$. One can compute, either manually or by SageMath \cite{sage}, that $$\psi'(G)=\dfrac{55}{171}<\dfrac{31}{77};\quad  \psi''(G)=\dfrac{55}{256}<\dfrac{31}{144};$$ Thus, the supersolvability of $G$ is not guaranteed by these parameters. However, we see that $$\psi''_2(G)=\dfrac{205}{4096}>\dfrac{85}{1728}.$$ Hence by Theorem \ref{thresholds-for-psi-double-prime}, we can say that $G$ is supersolvable. Similarly, the supersolvability of $\mathbb{Z}_7 \rtimes \mathbb{Z}_3$ can not be derived from their respective values of $\psi'$ and $\psi''$. However $\psi_3''$ does the job.
\end{example}

\begin{example}
Consider the group $G=\mathbb{Z}_{4}\times \mathbb{Z}_2$ of order $8$. One can compute, either manually or by SageMath \cite{sage}, that $$\psi'(G)=\dfrac{23}{43}<\dfrac{13}{21};\quad  \psi''(G)=\dfrac{23}{64}<\dfrac{13}{36};$$ Thus, the nilpotency of $G$ is not guaranteed by these parameters. However, we see that $$\psi''_2(G)=\dfrac{77}{512}>\dfrac{31}{216}.$$ Hence by Theorem \ref{thresholds-for-psi-double-prime}, we can say that $G$ is nilpotent. Similarly, the nilpotency of $\mathbb{Z}_{52} \times \mathbb{Z}_2$ can not be derived from their respective values of $\psi'$ and $\psi''$. However $\psi_3''$ does the job.
\end{example}

\section*{Acknowledgement}
The first author is supported by the CSIR PhD fellowship (File no. $08/0155(26311)/2026-EMR-I$), Govt. of India. The second author acknowledges the funding of DST-FIST Sanction no. $SR/FST/MS-I/2019/41$, Govt. of India. The third author is supported by the funding of UGC [NTA Ref. No. 211610129182], Govt. of India.

\subsection*{Statements and Declarations}
The authors have no relevant financial or non-financial interests to disclose. The authors have no competing interests to declare that are relevant to the content of this article. The authors also declare that there is no associated data.

All authors contributed to the study conception and design. Material preparation and analysis were performed by Sayanta Chatterjee, Angsuman Das and Sauvik Poddar. The first draft of the manuscript was written by Angsuman Das and all the authors commented on previous versions of the manuscript. All authors read and approved the final manuscript.

\bibliographystyle{abbrv}
\bibliography{Ref}

\section*{Appendix}

\noindent\textbf{Lemma A.} Let $k$ be a positive integer and $n\ge 4$. Let $G$ be one of the groups mentioned in Lemma \ref{Huppert's-classification-lemma}. Then
\begin{equation}\label{appendix-A-eqn}
\psi^{\prime\prime}_k(G)\le\frac{1+2^k+6\cdot 4^k}{8^{k+1}}=\psi^{\prime\prime}_k(Q_8).
\end{equation}

\pf We only show for $G=SD_{2^n}$, where $n\ge 4$. Let
$$L(n,k)\coloneqq\psi^{\prime\prime}_k(SD_{2^n})=\frac{2^{n(k+1)-1} + 2^{n+3k-1} + 2^{n+2k-2} - 2^{n+k-2} + 2^k - 1}{2^{n(k+1)}(2^{k+1} - 1)},$$ 
and
$$R(k)=\psi^{\prime\prime}_k(Q_8)=\frac{1+2^k+6\cdot 4^k}{8^{k+1}}.$$
Then
\begin{align*}
    L(n+1,k)&=\frac{2^{(n+1)(k+1)-1} + 2^{n+3k} + 2^{n+2k-1} - 2^{n+k-1} + 2^k - 1}{2^{(n+1)(k+1)}(2^{k+1}-1)}\\
    &= \frac{2^{nk+n+k} + 2^{n+3k} + 2^{n+2k-1} - 2^{n+k-1} + 2^k - 1}{2^{k+1}2^{n(k+1)}(2^{k+1}-1)}.
\end{align*}
Hence,
\begin{align*}
    \frac{L(n+1,k)}{L(n,k)}
    &=\frac{2^{nk+n+k} + 2^{n+3k} + 2^{n+2k-1} - 2^{n+k-1} + 2^k - 1}{2^{k+1}(2^{n(k+1)-1} + 2^{n+3k-1} + 2^{n+2k-2} - 2^{n+k-2} + 2^k - 1)} \\
    &=\frac{2^{nk+n+k} + 2^{n+3k} + 2^{n+2k-1} - 2^{n+k-1} + 2^k - 1}{2^{k}(2^{nk+n} + 2^{n+3k} + 2^{n+2k-1} - 2^{n+k-1} + 2^{k+1} - 2)}\\
    &=\frac{(2^{nk+n} + 2^{n+3k} + 2^{n+2k-1} - 2^{n+k-1} + 2^{k+1} - 2)+(-2^k+1+2^{nk+n+k}-2^{nk+n})}{2^{k}(2^{nk+n} + 2^{n+3k} + 2^{n+2k-1} - 2^{n+k-1} + 2^{k+1} - 2)}\\
    &=\frac{1}{2^k} \cdot \left( 1 + \frac{-2^k+1+2^{nk+n+k}-2^{nk+n}}{2^{nk+n} + 2^{n+3k} + 2^{n+2k-1} - 2^{n+k-1} + 2^{k+1} - 2} \right)\\
    &=\frac{1}{2^k} \cdot \left( 1 + \frac{-2^k+1+2^{nk+n}(2^k-1)}{2^{nk+n} + 2^{n+3k} + 2^{n+k-1}(2^k-1) + 2(2^k - 1)} \right)\\
    &< \frac{1}{2^k} \cdot \left( 1 + \frac{(2^{nk+n}-1)\cdot(2^k-1)}{2^{nk+n}} \right)\\
    &<\frac{1}{2^k} \cdot \left( 1+ (2^k-1) \right)\\
    &=1.
\end{align*} i.e.,
$$L(n+1,k)<L(n,k) \quad \text{for} \quad n,k\ge 1.$$
We show that $L(4,k)\le R(k)$. We have
$$L(4,k)=\frac{2^{4k+3}+2^{3k+3}+2^{2k+2}-2^{k+2}+2^k-1}{2^{4k+4}(2^{k+1}-1)}.$$
Substituting $2^k = x$ yields
$$L(4,k)=\frac{8x^4+8x^3+4x^2-4x+x-1}{16x^4(2x-1)}=\frac{8x^4+8x^3+4x^2-3x-1}{16x^4(2x-1)}.$$
Rewriting $R(k)$ in terms of $x$ yields 
$$R(k)=\frac{1+2^k+6(2^k)^2}{8(2^k)^3}=\frac{6x^2+x+1}{8x^3}.$$
Since $k\ge 1$, $x\ge 2$. We prove that for all $x\ge 2$,
\begin{equation}\label{eqn-appendix-A}
\frac{8x^4+8x^3+4x^2-3x-1}{16x^4(2x-1)}\le\frac{6x^2+x+1}{8x^3}.
\end{equation}
\eqref{eqn-appendix-A} is equivalent to
\begin{align*}
8x^4+8x^3+4x^2-3x-1&\le (6x^2+x+1)(4x^2-2x)\\
&=24x^4-12x^3+4x^3-2x^2+4x^2-2x\\
&=24x^4 - 8x^3 + 2x^2 - 2x,
\end{align*}
which is equivalent to
$$0 \le (24x^4 - 8x^3 + 2x^2 - 2x) - (8x^4 + 8x^3 + 4x^2 - 3x - 1)=16x^4 - 16x^3 - 2x^2 + x + 1,$$
which clearly holds for $x\ge 2$. Hence \eqref{eqn-appendix-A} holds. Thus $L(4,k)\le R(k)$ holds. Hence 
$$L(n,k)\le R(k) \quad \text{for} \quad n\ge 4 \quad \text{and} \quad k\ge 1,$$
which proves \eqref{appendix-A-eqn}.\qed

\medskip

\noindent\textbf{Lemma B.} Let $k$ be a positive integer and $n\ge 2$. Then
\begin{equation}\label{appendix-B-eqn}
\psi^{\prime\prime}_k(\mathbb{Z}_{2^{n-1}}\times\mathbb{Z}_2)\le\frac{1+3\cdot 2^k}{4^{k+1}}=\psi^{\prime\prime}_k(\mathbb{Z}_2\times\mathbb{Z}_2).
\end{equation}

\pf Let
$$L(n,k)=\psi^{\prime\prime}_k(\mathbb{Z}_{2^{n-1}}\times\mathbb{Z}_2)= \frac{2^{n(k+1)}+2^{2k+1}-2^k-1}{2^{n(k+1)}(2^{k+1}-1)},$$
and
$$R(k)=\psi^{\prime\prime}_k(\mathbb{Z}_2\times\mathbb{Z}_2)=\frac{1+3\cdot 2^k}{4^{k+1}}.$$
Then
\begin{align*}
    L(n+1,k)&= \frac{2^{(n+1)(k+1)}+2^{2k+1}-2^k-1}{2^{(n+1)(k+1)}(2^{k+1}-1)}\\
    &= \frac{2^{(n+1)(k+1)}+2^{2k+1}-2^k-1}{2^{k+1}2^{n(k+1)}(2^{k+1}-1)}.
\end{align*}
Hence,
\begin{align*}
    \frac{L(n+1,k)}{L(n.k)}
    &=\frac{2^{nk+n+k+1}+2^{2k+1}-2^k-1}{2^{k+1}(2^{n(k+1)} + 2^{2k+1} - 2^k - 1)}\\
    &=\frac{1}{2^{k+1}} \cdot \left( \frac{(2^{n(k+1)} + 2^{2k+1} - 2^k - 1)+(2^{nk+n+k+1}-2^{nk+n})}{2^{n(k+1)} + 2^{2k+1} - 2^k - 1} \right)\\
    &=\frac{1}{2^{k+1}} \cdot \left(1+ \frac{2^{nk+n+k+1}-2^{nk+n}}{2^{n(k+1)} + 2^{2k+1} - 2^k - 1} \right)\\
    &=\frac{1}{2^{k+1}} \cdot \left(1+ \frac{2^{nk+n}(2^{k+1}-1)}{2^{nk+n} + 2^{2k+1} - 2^k - 1} \right)\\
    &<\frac{1}{2^{k+1}} \cdot \left(1+ (2^{k+1}-1) \right)\\
    &=1.
\end{align*}
i.e.,
$$L(n+1,k) < L(n,k) \quad \text{for} \quad n,k\ge 1.$$
We show that $L(2,k)\le R(k)$. We have
$$L(2,k)=\frac{2^{2(k+1)}+2^{2k+1}-2^k-1}{2^{2(k+1)}(2^{k+1}-1)}.$$
Substituting $2^k=x$ yields
$$L(2,k)=\frac{4x^2+2x^2-x-1}{4x^2(2x-1)}=\frac{6x^2-x-1}{4x^2(2x-1)}=\frac{(3x+1)(2x-1)}{4x^2(2x-1)}.$$
Since $k\ge 1$, $x\ge 2$ and thus
$$L(2,k)=\frac{3x + 1}{4x^2}.$$
Now
$$R(k)=\frac{1+3\cdot 2^k}{4^{k+1}}=\frac{1+3x}{4x^2}.$$
Thus $L(2,k)=R(k)$. Hence
$$L(n,k)\le L(2,k)=R(k) \quad \text{for} \quad n\ge 2 \quad \text{and} \quad k\ge 1,$$
which proves \eqref{appendix-B-eqn}.\qed

\medskip

\noindent\textbf{Lemma C.} Let $k$ be a positive integer. For $G= S_5,\operatorname{SL}(2,5)$ or $\operatorname{PSL}(2,7)$,
\begin{equation}\label{appendix-eqn}
\psi^{\prime\prime}_k(G)<\frac{1+15\cdot 2^k+20\cdot 3^k+24\cdot 5^k}{60^{k+1}}=\psi^{\prime\prime}_k(A_5).
\end{equation}

\pf We only show for $G=\operatorname{PSL}(2,7)$. Then (\ref{appendix-eqn}) is equivalent to showing 
$$\frac{1+21\cdot 2^k+56\cdot 3^k+42\cdot 4^k+48\cdot 7^k}{168^{k+1}}<\frac{1+15\cdot 2^k+20\cdot 3^k+24\cdot 5^k}{60^{k+1}},$$
i.e.,
$$5\cdot 5^k+105\cdot 10^k+280\cdot 15^k+210\cdot 20^k+240\cdot 35^k<14\cdot 14^k+210\cdot 28^k+280\cdot 42^k+336\cdot 70^k,$$
i.e.,
\begin{equation}\label{eqn-2}
\scriptsize{(14^{k+1}-5^{k+1})+(210\cdot 28^k-210\cdot 20^k)+(280\cdot 42^k-280\cdot 15^k)+(240\cdot 70^k-240\cdot 35^k)+(96\cdot 70^k-105\cdot 10^k)>0.}
\end{equation}
Clearly,
$$14^{k+1}-5^{k+1}>0,$$
$$210\cdot 28^k-210\cdot 20^k>0,$$
$$280\cdot 42^k-280\cdot 15^k>0,$$
$$240\cdot 70^k-240\cdot 35^k>0.$$
Also since $7^k>\frac{35}{32}$ for all $k\ge 1$, we have $96\cdot 70^k-105\cdot 10^k>0$. Hence (\ref{eqn-2}) holds.\qed

\end{document}